\documentclass[11pt,twoside,letter]{article}
\usepackage[dvips]{graphicx}
\usepackage{enumerate}
\usepackage{amsfonts}
\usepackage{amsmath}
\usepackage{amssymb}
\usepackage{amsthm}

\newtheorem{theorem}{Theorem}[section]
\newtheorem{corollary}[theorem]{Corollary}
\newtheorem{lemma}[theorem]{Lemma}

\theoremstyle{definition}

\theoremstyle{remark}
\newtheorem{remark}[theorem]{Remark}

\newcommand{\bz}{\mathbb{Z}}

\newcommand{\bc}{\mathbb{C}}
\newcommand{\bt}{\mathbb{T}}

\usepackage{graphicx}

\newcommand{\Iidentity}{\hbox{\upshape \small1\kern-3.3pt\normalsize1}}

\begin{document}

\title{\bf\vspace{-39pt} Filters and Matrix Factorization}

\author{Palle E. T. Jorgensen \\ \small Department of Mathematics, The University of Iowa, \\
\small Iowa City, IA52242, USA  \\ \small jorgen@math.uiowa.edu \\
\\
Myung-Sin Song \\ \small Department of Mathematics and Statistics, Southern Illinois University Edwardsville\\
\small Edwardsville, IL62026, USA  \\ \small msong@siue.edu }

\date{}

\markboth{\footnotesize \rm \hfill P. E. T. Jorgensen AND M.-S. Song \hfill}
{\footnotesize \rm \hfill MATRIX FACTORIZATION \hfill}

\maketitle


\begin{abstract}
We give a number of explicit matrix-algorithms for analysis/synthesis in 
multi-phase filtering; i.e., the operation on discrete-time signals which 
allow a separation into frequency-band components, one for each of the ranges 
of bands, say $N$, starting with low-pass, and then corresponding filtering in 
the other band-ranges. If there are $N$ bands, the individual filters will be 
combined into a single matrix action; so a representation of the combined 
operation on all $N$ bands by an $N \times N$ matrix, where the corresponding 
matrix-entries are periodic functions; or their extensions to functions of a 
complex variable. Hence our setting entails a fixed $N \times N$ matrix over a 
prescribed algebra of functions of a complex variable. In the case of 
polynomial filters, the factorizations will always be finite. A novelty here 
is that we allow for a wide family of non-polynomial filter-banks.

Working modulo $N$ in the time domain, our approach also allows for a natural 
matrix-representation of both down-sampling and up-sampling. The implementation 
encompasses the combined operation on input, filtering, down-sampling, 
transmission, up-sampling, an action by dual filters, and synthesis, merges 
into a single matrix operation. Hence our matrix-factorizations break down the 
global filtering-process into elementary steps. To accomplish this, we offer a 
number of adapted matrix factorization-algorithms, such that each factor in 
our product representation implements in a succession of steps the filtering 
across pairs of frequency-bands; and so it is of practical significance in 
implementing signal processing, including filtering of digitized images. Our 
matrix-factorizations are especially useful in the case of the processing a 
fixed, but large, number of bands. 
\vspace{5mm} \\

\vspace{3mm}

\end{abstract}

\section{Introduction}
\label{sec:0}
Our purpose is to establish factorization of matrices $M_{N}(\mathcal{A})$ 
over certain rings $\mathcal{A}$ of functions, among them the ring of 
polynomials, and the  $L^{\infty}$ functions on the circle group $\mathbb{T}$. 
An equivalent 
formulation is the study of functions on $\mathbb{T}$ which take values in the 
$N \times N$ scalar matrices. The general setting is as follows: Fix $N$, and 
consider the group $SL_{N}(\mathcal{A})$ where the ``$S$" is for determinant 
$= 1$. The object is then to factor arbitrary elements in 
$SL_{N}(\mathcal{A})$ as alternating products of upper and lower triangular 
matrix functions; equivalently, upper and lower triangular elements in 
$M_{N}(\mathcal{A})$ with the constant $1$ in the diagonal.

In digital signal or image-processing one makes use of subdivisions of various 
families of signals into frequency bands. This is of relevance in modern-day 
wireless signal and image processing, and the choice of a number $N$ of 
frequency bands may vary from one application to the next.

There is a certain representation theoretic framework which has proved 
successful: one builds a representation of the basic operations on signals, 
filtering, down-sampling (in the complex frequency variable), up-sampling, and 
dual filter. These operations get represented by a system of operators in 
Hilbert spaces of states, say $\mathcal{H}$.

A multiresolution (see Fig. 1) then takes the form of a family of closed 
subspaces in $\mathcal{H}$.  In this construction, ``non-overlapping frequency 
bands" correspond to orthogonal subspaces in $\mathcal{H}$; or equivalently to 
systems of orthogonal projections. Since the different frequency bands must 
exhaust the signals for the entire system, one looks for orthogonal 
projections which add to the identity operator in $\mathcal{H}$. This leads 
to the study of certain representations of the Cuntz algebra $\mathcal{O}_N$, 
details below. Since time/frequency-analysis is non-commutative, one is 
further faced with a selection of special families of commuting orthogonal 
projections. When these iteration schemes (repeated subdivision sequences) are 
applied to the initial generators, one arrives at new bases and frames; and, 
in other applications, to wavelet families as recursive scheme. 

Our study of iterated matrix-factorizations are motivated by such questions 
from signal processing, and arising in multi-resolution analyses. In this 
case, elements in the group $SL_{N}(\mathcal{A})$ of matrix-functions act on 
vector-functions $f$ in a complex frequency variable, where the components in 
$f$ correspond to a specified system of $N$ frequency-bands. When a 
matrix-factorization is established, then the action of the respective upper 
and lower triangular elements in $M_{N}(\mathcal{A})$ are especially simple, 
in that a lower triangular filter filters a low band, and then adds it to one 
of the higher bands; and similarly for the action of upper triangular matrix 
functions.

Our analysis depend on a certain representation of the Cuntz algebra 
$\mathcal{O}_N$,  where $\mathcal{O}_N$ is an algebra generated by the basic 
operations on signal representations, filtering, down-sampling (in the complex 
frequency variable), up-sampling, and dual filter; see Fig 1. 




\[
  \sum_{n}b_{n}z^{nN}=b_{0}+b_{1}z^{N}+b_{2}z^{2N}+\cdots \text{ ;}
\]
so
\[
  c_{n}=
  \begin{cases}
    b_{n/N}  \quad \text{if } N | n \\
    0   \quad \text{if } N \nmid n
  \end{cases}
\]


\[
  \frac{1}{N}\sum_{w \in \mathbb{T}, w^{N}=z}b_{n}w^{n}=b_{0}+b_{N}z
  +b_{2N}z^{2}+\cdots
\]



\section{Factorization Algorithm}
\label{sec:2}

In order to illustrate our use of representations of the Cuntz algebra 
$O_{N}$ in algorithms for factorization, we begin with the case of $N=2$. The 
skeleton of these algorithms has three basic steps which we now outline.

\subsection*{The Algorithm}
Given
\[
  \begin{pmatrix}
    A & B \\
    C & D
  \end{pmatrix}
  \in SL_{2}(\mathcal{F})
\]
where $\mathcal{F}$ is some fixed ring of functions defined on a subset
$\Omega \subset \mathbb{C}$ such that $\mathbb{T} \subset \Omega$.

\subsubsection*{Step 1:}
Given
\[
  \mathcal{A}=
  \begin{pmatrix}
    A & B \\
    C & D
  \end{pmatrix},
  \quad AD-BC\equiv 1 \quad \text{on $\mathbb{T}$,}
\]
and set
\begin{equation}
\label{eq:2.1}
  \mathcal{A}(z^{2})
  \begin{pmatrix}
    1 \\
    z
  \end{pmatrix}
  =
  \begin{pmatrix}
    A(z^{2})+zB(z^{2}) \\
    C(z^{2})+zD(z^{2})
  \end{pmatrix}.
\end{equation}
Let $S_{i}$, $i=0,1$ be  

\begin{equation}
\label{eq:2.2}
  \begin{cases}
    S_{0}f(z)=f(z^{2}) \\
    S_{1}f(z)=zf(z^{2})
  \end{cases}
\end{equation}
For the corresponding adjoint operators we therefore get:
\begin{equation}
\label{eq:2.3}
  \begin{cases}
    S_{0}^{*}f(z)=\frac{1}{2}\sum_{\omega^{2}=z}f(\omega) \\
    S_{1}^{*}f(z)=\frac{1}{2}\sum_{\omega^{2}=z}\overline{\omega}f(\omega)
 \end{cases}
\end{equation}
where the summation in (\ref{eq:2.2}), (\ref{eq:2.3}) are over points
$z, \omega \in \mathbb{T}$.

Then $(S_{i})_{i=0,1}$ are isometries in $L^{2}(\mathbb{T})$, and
$S_{i}^{*}S_{j}=\delta_{i,j}I$, $\sum_{i=0}^{1}S_{i}S_{j}^{*}=I$ where $I$
denotes the identity operator in the Hilbert space $L^{2}(\mathbb{T})$.
We will want $\mathcal{F}$ to be a ring of meromorphic functions, such that
they are determined by their values on
$\mathbb{T}=\{z \in \mathbb{C}, |z|=1\}$;
or we are simply working with functions on $\mathbb{T}$.

\subsubsection*{Step 2:}
Find functions $L$ such that
\begin{equation}
\label{eq:2.4}
  \begin{pmatrix}
    l & 0 \\
    L & 1
  \end{pmatrix}
  \mathcal{A}_{new}
  =\mathcal{A}.
\end{equation}
Solution: Apply (\ref{eq:2.4}) to
\[
  \begin{pmatrix}
    1 \\
    z
  \end{pmatrix},
\]
and set
\[
  \mathcal{A}_{new}(z^{2})
  \begin{pmatrix}
    1 \\
    z
  \end{pmatrix}
  =
  \begin{pmatrix}
    f_{0}(z) \\
    f_{1}(z)
  \end{pmatrix};
\]
then
\begin{equation}
\label{eq:2.5}
  \begin{cases}
    f_{0}=A(z^{2})+zB(z^{2}) \\
    L(z^{2})f_{0}(z)+f_{1}(z)=C(z^{2})+zD(z^{2}).
  \end{cases}
\end{equation}
Apply $S_{i}^{*}$, $i=0,1$, to (\ref{eq:2.5})
\begin{equation}
\label{eq:2.6}
  \begin{cases}
    S_{0}^{*}f_{0}=A, \quad S_{1}^{*}f_{0}=B \\
    LS_{0}^{*}f_{0}+S_{0}^{*}f_{1}=C \\
    LS_{1}^{*}f_{0}+S_{1}^{*}f_{1}=D. \\
    \Rightarrow L=\frac{C-S_{0}^{*}f_{1}}{A}; \quad
    L=\frac{D-S_{1}^{*}f_{1}}{B}.
  \end{cases}
\end{equation}

\begin{corollary}
  $A(S_{1}^{*}f_{1})-B(S_{0}^{*}f_{1})=1$.
\end{corollary}
\begin{proof}
  Consider (\ref{eq:2.6}) with $det \mathcal{A}=1$.
  \[
    \mathcal{A}=
    \begin{pmatrix}
      A & B \\
      C & D
    \end{pmatrix}
  \]
  with $AD-BC=1$.
  So
  \[
    \begin{pmatrix}
      A & B \\
      S_{0}^{*}f_{1} & S_{1}^{*}f_{1}
    \end{pmatrix}
    \in SL_{2}(\mathcal{F}).
  \]
  We now assume $S_{i}\mathcal{F} \subset \mathcal{F}$, and
  $S_{i}^{*}\mathcal{F} \subset \mathcal{F}$, for all $i=0,1$.
\end{proof}

\subsubsection*{Step 3:}
Having form $L$, from (\ref{eq:2.4}) we get
\[
  \mathcal{A}_{new} =
  \begin{pmatrix}
    l & 0 \\
    -L & 1
  \end{pmatrix}
  \begin{pmatrix}
      A & B \\
      C & D
  \end{pmatrix}
  =
  \begin{pmatrix}
      A & B \\
      -LA+C & -LB+D
  \end{pmatrix}
  =
  \begin{pmatrix}
      A & B \\
      S_{0}^{*}f_{1} & S_{1}^{*}f_{1}
  \end{pmatrix}
\]

\subsubsection*{Step 4:}
\[
  \begin{pmatrix}
    l & U \\
    0 & 1
  \end{pmatrix}
  \mathcal{A}_{up}
  =\mathcal{A}_{new}.
\]
Set
\[
  \mathcal{A}_{up}(z^{2})
  \begin{pmatrix}
    1 \\
    z
  \end{pmatrix}
  =
  \begin{pmatrix}
    g_{0}(z) \\
    g_{1}(z)
  \end{pmatrix};
\]
and we get
\[
  \begin{cases}
    g_{0}(z)+U(z^{2})g_{1}(z)=A(z^{2})+B(z^{2})z \\
    g_{1}(z)=(S_{0}^{*}f_{1})(z^{2})+(S_{1}^{*}f_{1})(z^{2})z
  \end{cases}.
\]
Apply $S_{i}^{*}$, $i=0,1$
\[
  \Rightarrow
  \begin{cases}
  S_{0}^{*}g_{0}+US_{0}^{*}g_{1}=A, \quad S_{1}^{*}g_{0}+US_{1}^{*}g_{1}=B \\
  S_{0}^{*}g_{1}=S_{0}^{*}f_{1}, \quad S_{1}^{*}g_{1}=S_{1}^{*}f_{1} \\
  \Rightarrow U=\frac{A-S_{0}^{*}g_{0}}{S_{0}^{*}f_{1}} \quad \text{and} \quad
  U=\frac{B-S_{1}^{*}g_{0}}{S_{1}^{*}f_{1}}
  \end{cases}
\]
and continue.

\[
  S_{0}^{*}f_{0}=A, \quad S_{1}^{*}f_{0}=B
\]
\[
  LS_{0}^{*}f_{0}+S_{0}^{*}f_{0}=C
\]
\[
  LS_{1}^{*}f_{0}+S_{1}^{*}f_{0}=D
\]
\[
  L=\frac{C-S_{0}^{*}f_{1}}{A}=\frac{D-S_{1}^{*}f_{1}}{B}
\]
\[
  A(D-S_{1}^{*}f_{1})=B(C-S_{0}^{*}f_{1})
\]
\[
  1=AS_{1}^{*}f_{1}-BS_{0}^{*}f_{1}
\]
\[
  A_{new}^{(1)} = 
  \begin{pmatrix}
    1 & 0 \\
    -L & 1
  \end{pmatrix}
  \begin{pmatrix}
    A & B \\
    C & D
  \end{pmatrix} 
  =
  \begin{pmatrix}
    A & B \\
    -LA+C & -LB+D
  \end{pmatrix}
  =
  \begin{pmatrix}
    A & B \\
    S_{0}^{*}f_{1} & S_{1}^{*}f_{1}
  \end{pmatrix} 
\]
so
\[
  \begin{pmatrix}
    1 & 0 \\
    L & 1
  \end{pmatrix}
  \begin{pmatrix}
    A & B \\
    S_{0}^{*}f_{1} & S_{1}^{*}f_{1}
  \end{pmatrix}
  =
  \begin{pmatrix}
    A & B \\
    C & D
  \end{pmatrix}.
\]


\section{Factorization Cases}
\label{sec:2}
In the infinite-dimensional group 
$SL_{2}(L^{\infty}(\mathbb{T}))$, consider elements 
$\mathcal{A}$ with factorization as in (\ref{eq:3.1}):
\[
  \mathcal{A}=
  \begin{pmatrix}
    A & B \\
    C & D
  \end{pmatrix}
\]

\begin{equation}
\label{eq:3.1}
  \mathcal{A}=
  \begin{pmatrix}
    1 & 0 \\
    L & 1
  \end{pmatrix}
  \mathcal{A}^{(1)},
  \quad L \in L^{\infty}(\mathbb{T}), \quad \mathcal{A}^{(1)} \in
  SL_{2}(L^{\infty}(\mathbb{T}))
\end{equation}
Optimal
\[
  \mathcal{A}^{(1)}(z^{2})
  \begin{pmatrix}
    1 \\
    z
  \end{pmatrix}
  =
  \begin{pmatrix}
    f_{0} \\
    f_{1}
  \end{pmatrix}
\]
\begin{equation}
\label{eq:3.2}
  \begin{cases}
    A(z^{2})+zB(z^{2})=f_{0}(z) \\
    C(z^{2})+zD(z^{2})=L(z^{2})f_{0}(z)+f_{1}(z)
  \end{cases}
  \quad \{S_{i}\}_{i=0} \in REP(\mathcal{O}_{2}, L^{2}(\mathbb{T}))
\end{equation}

\begin{equation}
\label{eq:3.3}
  \begin{cases}
    S_{0}^{*}f_{0}=A, S_{1}^{*}f_{0}=B \\
    LS_{0}^{*}f_{0}+S_{0}^{*}f_{1}=C \\
    LS_{1}^{*}f_{0}+S_{1}^{*}f_{1}=D
  \end{cases}
\end{equation}
\begin{equation}
\label{eq:3.4}
  \iff
  \begin{cases}
    S_{0}^{*}f_{1}=C-LA  \\
    S_{1}^{*}f_{1}=D-LB
  \end{cases}
\end{equation}
\begin{equation}
\label{eq:3.5}
  \Rightarrow
  f_{1}=(S_{0}S_{0}^{*}+S_{1}S_{1}^{*})f_{1}=S_{0}(C-LA)+S_{1}(D-LB)
\end{equation}
\[
  \begin{pmatrix}
    A & B \\
    C & D
  \end{pmatrix}
  \longrightarrow
  \begin{pmatrix}
    A & B \\
    S_{0}^{*}f_{1} &  S_{1}^{*}f_{1}
  \end{pmatrix}
\]

Since $S_{i}$ is isometric for $i=1,2$.
\begin{equation}
\label{eq:3.6}
  \|f_{1}\|^{2}=\|C-LA\|^{2}+\|D-LB\|^{2} \quad \text{where $\|\cdot\|$ is
  the $L^{2}(\mathbb{T})-$norm.}
\end{equation}

\begin{equation}
\label{eq:3.7}
  \langle u, v \rangle=\int_{\mathbb{T}}\overline{u}v \quad \text{with
  respect to Haar measure on $\mathbb{T}$.}
\end{equation}
So any functions
\begin{equation}
\label{eq:3.8}
  \mathcal{A}=
  \begin{pmatrix}
    1 & 0 \\
    L & 1
  \end{pmatrix}
  \mathcal{A}^{(1)}
\end{equation}
we pick the one with $f_{1}$ attaching its minimum in (\ref{eq:3.6})
\begin{equation}
\label{eq:3.7-1}
  \inf\{ (\ref{eq:3.6})| \text{factorization } (\ref{eq:3.8}) \text{ holds}\}
\end{equation}
Calculating $L$ on $\mathcal{A}$
\[
  L_{M}(\epsilon)=L+\epsilon M, \quad M \in L^{\infty}(\mathbb{T}).
\]

\begin{equation}
\label{eq:3.8-1}
  \frac{d}{d\epsilon}\bigg\vert_{\epsilon=0} (\ref{eq:3.6})=0 \quad
  \text{ at a minimum.}
\end{equation}

\begin{align*}
  &\iff \\
  &\langle MA, C-LA \rangle + \langle C-LA, MA \rangle +
  \langle MB, D-LB \rangle + \langle D-LB, MB \rangle \\
  &=Re(\langle MA, C-LA \rangle+\langle MB, D-LB \rangle)=0 \quad
  \forall M \in L^{\infty}(\mathbb{T}).
\end{align*}

\begin{equation}
\label{eq:3.9-1}
  \overline{A}(C-LA)+\overline{B}(D-LB)=0 \quad \text{pointwise a. e. on
  $\mathbb{T}$.}
\end{equation}
Set $det\mathcal{A}=1$,
\[
  \|A\|^{2}+\|B\|^{2}>0 \quad \text{a. e. on $\mathbb{T}$.}
\]
So
\begin{equation}
\label{eq:3.10}
  L= \frac{\overline{A}C+\overline{B}D}{|A|^{2}+|B|^{2}} \quad
  \text{pointwise a. e. $\mathbb{T}$.}
\end{equation}
Solving for matrices $\mathcal{A}^{(1)}$ in (\ref{eq:3.8-1}), we get
\[
  \mathcal{A}^{(1)}=
  \begin{pmatrix}
    1 & 0 \\
    -L & 1
  \end{pmatrix}
  \begin{pmatrix}
    A & B \\
    C & D
  \end{pmatrix}
  =
  \begin{pmatrix}
    A & B \\
    C-LA & D-LB
  \end{pmatrix}.
\]
So
\[
  \mathcal{A}=
  \begin{pmatrix}
    1 & 0 \\
    L & 1
  \end{pmatrix}
  \mathcal{A}^{(1)}
\]
With the above $L$ in (\ref{eq:3.10}) we see that
\[
  \mathcal{A}=
  \begin{pmatrix}
    1 & 0 \\
    L & 1
  \end{pmatrix}
  \mathcal{A}^{(1)}
\]
is the \underline{optimal} factorization with a lower matrix as a
left-factor.

\begin{corollary}
\label{C:3.2}
Given
  \[
    \begin{pmatrix}
      A & B \\
      C & D
    \end{pmatrix}
    \in GL_{2}(L^{\infty}(\mathbb{T}));
  \]
then the optimal solution (\ref{eq:3.10}) to the factorization problem
\begin{equation}
\label{eq:3.17}
  \begin{pmatrix}
    A & B \\
    C & D
  \end{pmatrix}
  =
  \begin{pmatrix}
    1 & 0 \\
    L & 1
  \end{pmatrix}
  \begin{pmatrix}
    A & B \\
    S_{0}^{*}f_{1} & S_{1}^{*}f_{1}
  \end{pmatrix}
\end{equation}
has the matrix
\[
  \begin{pmatrix}
    A & B \\
    S_{0}^{*}f_{1} & S_{1}^{*}f_{1}
  \end{pmatrix}
\]
on the right hand side in (\ref{eq:3.17}) orthogonal, i.e.,
\begin{equation}
\label{eq:3.18}
  \overline{A}(S_{0}^{*}f_{1})+\overline{B}(S_{1}^{*}f_{1}) \equiv 0
  \quad \text{on $\mathbb{T}$.}
\end{equation}
\end{corollary}
\begin{proof}
When the function $L$ in (\ref{eq:3.10}) is used in the computation of
\[
  \begin{pmatrix}
    A & B \\
    S_{0}^{*}f_{1} & S_{1}^{*}f_{1}
  \end{pmatrix},
\]
we see that for any $z \in \mathbb{T}$,
$((S_{0}^{*}f_{1})(z), (S_{1}^{*}f_{1})(z))$ in $\mathbb{C}$ is in the
orthogonal complement of $(A(z), B(z))$; indeed with (\ref{eq:3.10}) we get
\begin{align*}
  &\overline{A}(S_{0}^{*}f_{1})+\overline{B}(S_{1}^{*}f_{1}) \\
  &=\overline{A}\left(C-\frac{\overline{A}C+\overline{B}D}{|A|^{2}+|B|^{2}}A\right)+
  \overline{B}\left(D-\frac{\overline{A}C+\overline{B}D}{|A|^{2}+|B|^{2}}B\right) \\
  &=\overline{A}C+\overline{B}D-(\overline{A}C+\overline{B}D)\equiv 0;
\end{align*}
i.e., a pointwise identity for functions on $\mathbb{T}$.
\end{proof}

\begin{corollary}
\label{C:3.1}
  If $\mathcal{A} \in SU(L^{\infty}(\mathbb{T}))$ (i.e., unitary) then $L$ in
(\ref{eq:3.10}) is $0$ and so $\mathcal{A}=\mathcal{A}^{(1)}$ so the
factorization steps.
\end{corollary}
\begin{proof}
\[
  \mathcal{A}=
  \begin{pmatrix}
    A & B \\
    C & D
  \end{pmatrix},
\]
so unitary makes that the rows are orthogonal $\overline{A}C+\overline{B}D=0$
in the inner product on $\mathbb{C}^{2}$
\[
  \langle z, w \rangle = \overline{z_{1}}w_{1}+\overline{z_{2}}w_{2}
\]
and $|A|^{2}+|B|^{2}=1$.
\end{proof}
\begin{equation}
\label{eq:3.10-1}
  \mathcal{A}^{(1)}=
  \begin{pmatrix}
    A & B \\
    S_{0}^{*}f_{1} & S_{1}^{*}f_{1}
  \end{pmatrix}
\end{equation}
using $\longrightarrow$
\[
  \begin{pmatrix}
    S_{0}^{*}g_{0} & S_{1}^{*}g_{0} \\
    S_{0}^{*}f_{1} & S_{1}^{*}f_{1}
  \end{pmatrix}.
\]
Note this using the repeated on any
$\mathcal{A}^{(1)} \in SL_{2}(L^{\infty}(\mathbb{T}))$ each time pick $L$ such
that the infimum in (\ref{eq:3.6}) is attained.

With the same argument, we factor matrix
\[
  \begin{pmatrix}
    1 & U \\
    0 & 1
  \end{pmatrix}
  \quad U \in L^{\infty}(\mathbb{T})
\]
\begin{equation}
\label{eq:3.11}
  \mathcal{A}=
  \begin{pmatrix}
    1 & U \\
    0 & 1
  \end{pmatrix}
  \mathcal{A}^{(2)}, \quad \mathcal{A}^{(2)} \in SL_{2}(L^{\infty}(\mathbb{T})).
\end{equation}
Set
\begin{equation}
\label{eq:3.12}
  \begin{pmatrix}
    g_{0} \\
    g_{1}
  \end{pmatrix}
  =\mathcal{A}^{(2)}(z^{2})
  \begin{pmatrix}
    1 \\
    z
  \end{pmatrix}
\end{equation}
\[
  \begin{cases}
    A(z^{2})+zB(z^{2})=g_{0}+U(z^{2})g_{1} \\
    C(z^{2})+zD(z^{2})=g_{1}
  \end{cases}
\]
\begin{equation}
\label{eq:3.13}
  \begin{cases}
    A=S_{0}^{*}g_{0}+US_{0}^{*}g_{1} \\
    B=S_{1}^{*}g_{0}+US_{1}^{*}g_{1} \\
    C=S_{0}^{*}g_{1}, D=S_{1}^{*}g_{1}
  \end{cases}
  \quad S_{0}^{*}g_{0} = A-UC, \quad S_{1}^{*}g_{0}=B-UD
\end{equation}
\[
  g_{0}=S_{0}S_{0}^{*}g_{0}+S_{1}S_{1}^{*}g_{1}
  =S_{0}(A-UC)+S_{1}(B-UD)
\]
\begin{equation}
\label{eq:3.14}
  \|g_{0}\|^{2}=\|A-UC\|^{2}+\|B-UD\|^{2}
\end{equation}
such that (\ref{eq:3.11}) holds.
\[
  \begin{pmatrix}
    A & B \\
    C & D
  \end{pmatrix}
  \longrightarrow
  \begin{pmatrix}
    A-UC & B-UD \\
    C & D
  \end{pmatrix},
  \quad S_{0}^{*}g_{0} = A-UC, \quad S_{1}^{*}g_{0}=B-UD.
\]
Pick $U$ such that
\[
  \overline{C}(A-UC)+\overline{D}(B-UD)=0
\]
\begin{equation}
\label{eq:3.15}
  U=\frac{\overline{C}A+\overline{D}B}{|C|^{2}+|D|^{2}}
\end{equation}
\begin{equation}
\label{eq:3.16}
  \mathcal{A}^{(2)}=
  \begin{pmatrix}
    S_{0}^{*}g_{0} & S_{1}^{*}g_{0} \\
    C & D
  \end{pmatrix}
\end{equation}
in (\ref{eq:3.11}).

If
\[
  \mathcal{A}=
  \begin{pmatrix}
    A & B \\
    C & D
  \end{pmatrix}
  \in SL_{2}(L^{\infty}(\mathbb{T}))
\]
then
\[
  U=\frac{\overline{C}A+\overline{D}B}{|C|^{2}+|D|^{2}} =0.
\]
See (\ref{eq:3.15}) so the factorization
\[
  \mathcal{A}=
  \begin{pmatrix}
    1 & U \\
    0 & 1
  \end{pmatrix}
  \mathcal{A}^{(2)}
\]
in (\ref{eq:3.11}) is then, $U=0 \Rightarrow \mathcal{A}=\mathcal{A}^{(2)}$.
Then following factorization results:
\[
  \mathcal{A}=(\prod(lower)(upper))SL_{2}(L^{\infty}(\mathbb{T}))
\]
\begin{equation}
\label{eq:3.16-1}
  \begin{pmatrix}
    A & B \\
    C & D
  \end{pmatrix}
  \underset{\text{factor out lower matrix on the left}}{\longrightarrow}
  \begin{pmatrix}
    A & B \\
    S_{0}^{*}f_{1} & S_{1}^{*}f_{1}
  \end{pmatrix}
\end{equation}
\[
  \underset{\text{factor out upper matrix on the left}}{\longrightarrow}
  \begin{pmatrix}
    S_{0}^{*}g_{0} & S_{1}^{*}g_{0} \\
    S_{0}^{*}f_{1} & S_{1}^{*}f_{1}
  \end{pmatrix}.
\]
Or equivalently,
\begin{equation}
\label{eq:3.17}
  \begin{pmatrix}
    A & B \\
    C & D
  \end{pmatrix}
  =
  \begin{pmatrix}
    1 & 0 \\
    L & 1
  \end{pmatrix}
  \begin{pmatrix}
    1 & U \\
    0 & 1
  \end{pmatrix}
  \begin{pmatrix}
    S_{0}^{*}g_{0} & S_{1}^{*}g_{0} \\
    S_{0}^{*}f_{1} & S_{1}^{*}f_{1}
  \end{pmatrix}.
\end{equation}

\begin{corollary}
\label{C:3.2}
Consider $\mathcal{A}\in SL_{2}(L^{\infty}(\mathbb{T}))$, and the 
factorization
\begin{equation}
\label{eq:3.18}
  \mathcal{A}=
  \begin{pmatrix}
    1 & 0 \\  
    L_{1} & 1
  \end{pmatrix}
  \begin{pmatrix}
    1 & U_{1} \\  
    0 & 1
  \end{pmatrix}
  \cdots 
  \begin{pmatrix}
    1 & 0 \\  
    L_{p} & 1
  \end{pmatrix}
  \begin{pmatrix}
    1 & U_{p} \\  
    0 & 1
  \end{pmatrix}
  \begin{pmatrix}
    S_{0}^{*}g_{0} & S_{1}^{*}g_{0} \\  
    S_{0}^{*}f_{1} & S_{1}^{*}f_{1}
  \end{pmatrix}
\end{equation}
resulting from an iteration of the algorithm from (\ref{eq:3.17}).  Then 
the last factor in (\ref{eq:3.18}) is of diagonal form if and only if the 
following hold:
There are functions $\varphi, \psi \in L^{2}(\mathbb{T})$ such that 
\begin{equation}
\label{eq:3.19}
  g_{0}(z)=\varphi(z^{2}), \quad \text{and} \quad f_{1}(z)=z\psi(z^{2});
\end{equation}
and, in this case, the last factor in (\ref{eq:3.18}) is as follows:
\begin{equation}
\label{eq:3.20}
  \begin{pmatrix}
    S_{0}^{*}g_{0} & S_{1}^{*}g_{0} \\  
    S_{0}^{*}f_{1} & S_{1}^{*}f_{1}
  \end{pmatrix}
=
  \begin{pmatrix}
    \varphi & 0 \\  
    0 & \psi
  \end{pmatrix}.
\end{equation}
\end{corollary}
\begin{proof}
This follows from (\ref{eq:3.17}), and the Cuntz-relations:
\begin{equation}
\label{eq:3.21}
  S_{i}^{*}S_{j}=\delta_{i,j}, \quad \sum_{i}S_{i}S_{i}^{*}=I.
\end{equation}

\end{proof}

\subsection{Factorizations}
\label{sec:4.1.1}

We fix a value of $N > 1$ (i.e., the given number of frequency bands), and we 
begin with the formula for a canonical system of $N$ isometries  $S_{i}$ which 
define an associated representation of the Cuntz algebra $O_{N}$.  Said 
differently: The system of isometries $\{S_{i} \}$ satisfies the Cuntz 
relations with reference to the Hilbert space $L^{2}(\mathbb{T})$ where 
$\mathbb{T}$ is the circle group (one-torus) with its normalized invariant 
Haar measure. When the value of $N$ is fixed, then the multi-resolution 
filters will then take the form of $N \times N$ matrix functions; the matrix 
entries might be polynomials, or, more generally, functions from 
$L^{\infty}(\mathbb{T})$. Hence the questions about matrix factorization 
depends on the context. In the case of polynomial entries we will make use of 
degree, but this is not available for the more general case of entries from the 
algebra $L^{\infty}(\mathbb{T})$. In every one of the settings, we develop 
factorization algorithms, and the particular representation of the Cuntz 
algebra will play an important role. 

The standard representation of $O_{N}$, which we will use below, is given by 
the system of isometries $\{S_{j}\}$ as follows:

\begin{equation}
\label{E:3.10}
  (S_{j}\varphi)(z)=f_{j}(z)\varphi(z^{N}).
\end{equation}
\begin{lemma}
\label{L:3.10}
\cite{JoSo10} Let $N\in \bz_{+}$ be given and let $F=(f_{j})_{j\in \bz_{+}}$ be a function
system.  Then $F \in \mathcal{O}\mathcal{F}_{N}$ if and only if the operators
$S_{j}$ \ref{E:3.10}) satisfy
\begin{equation}
\label{eq:3.11}
  S_{j}^{*}S_{k}=\delta_{j,k}I
\end{equation}
\begin{equation}
\label{eq:3.12}
  \sum_{j \in \bz_{N}}S_{j}S_{j}^{*}=I,
\end{equation}
where $I$ denotes the identity operator in $\mathcal{H}=L^{2}(\bt)$.
\end{lemma}

We say that the isometries $\{ S_{j} \}_{j \in \mathbb{Z}_{N}}$ define a
representation of the Cuntz-algebra $\mathcal{O}_{N}$,
$(S_{j}) \in Rep(\mathcal{O}_{N}, L^{2}(\mathbb{T}))$.

\begin{lemma}
\label{L:4.1}
\cite{JoSo10} Let $N\in\bz_{+}$ be fixed, $N>1$, and let $A=(A_{j,k})$ be an $N \times N$
matrix-function with $A_{j,k} \in L^{2}(\bt)$.  Then the following two
conditions are equivalent:
\begin{enumerate} [(i)]
\item For $F=(f_{j})\in \mathcal{F}_{2}(N)$, we have $F(z)=A(z^{N})b(z)$.
\item $A_{i,j}=S_{j}^{*}f_{i}$ where the operators $S_{i}$ are from the
  Cuntz-relations (\ref{eq:3.11}, \ref{eq:3.12}).
\end{enumerate}
\end{lemma}
\begin{proof}
(i) $\Rightarrow$ (ii). Writing out the matrix-operation in (i), we get
\begin{equation}
\label{E:4.2}
  f_{i}(z)=\sum_{j}A_{i,j}(z^{N})z^{j}=\sum_{j}(S_{j}A_{i,j})(z).
\end{equation}
Using $S_{j}^{*}S_{k}=\delta_{j,k}I$, we get $A_{i,j}=S_{j}^{*}f_{i}$ which
is (ii).

Conversely, assuming (ii) and using $\sum_{j}S_{i}S_{j}^{*}=I$, we get
$\sum_{j}S_{j}A_{i,j}=f_{i}$ which is equivalent to (i) by the computation in
(\ref{E:4.2}) above.
\end{proof}

\begin{theorem}
\label{T:4.2.5}
(Sweldens \cite{SwRo91}, \cite{JoSo10})
Let $A \in SL_{2}\text{(pol)}$, then there are $l, p \in \bz_{+}$,
$K \in \bc \setminus \{0\}$ and polynomial functions $U_{1}, \ldots, U_{p}$,
$L_{1}, \ldots, L_{p}$ such that
\begin{equation}
\label{E:4.2.12}
  A(z)=z^{l}
  \begin{pmatrix}
    K & 0 \\
    0 & K^{-1}
  \end{pmatrix}
  \begin{pmatrix}
    1 & U_{1}(z) \\
    0 & 1
  \end{pmatrix}
  \begin{pmatrix}
    1 & 0 \\
    L_{1}(z) & 1
  \end{pmatrix}
  \cdots
  \begin{pmatrix}
    1 & U_{p}(z) \\
    0 & 1
  \end{pmatrix}
  \begin{pmatrix}
    1 & 0 \\
    L_{p}(z) & 1
  \end{pmatrix}.
\end{equation}
\end{theorem}

The filter algorithm corresponding to the matrix-factorization in 
(\ref{E:4.2.12}) is as follows:
And in steps:

\begin{remark}
\label{R:4.2.6}
\cite{JoSo10} Note that if
\[
  \begin{pmatrix}
    \alpha & \beta \\
    \gamma & \delta
  \end{pmatrix}
  \in SL_{2}\text{(pol)},
\]
then one of the two functions $\alpha(z)$ or $\delta(z)$ must be a monomial.
\end{remark}

\subsection{The $2 \times 2$ case: Polynomials}
\label{sec:5.1}
\cite{JoSo10} To highlight the general ideas, we begin with some details worked out in the
$2 \times 2$ case; see equation (\ref{eq:3.16-1}).

To get finite algorithms, we should assume in the present subsection that the
matrix-entries are polynomials.

First note that from the setting in Theorem \ref{T:4.2.5}, we may assume
that matrix entries have the form $f_{H}(z)$ as in section \ref{sec:2} but
with $H \subset \{0,1,2, \cdots\}$, i.e., $f_{H}(z)=a_{0}+a_{1}z+ \cdots$.
This facilitates our use of the Euclidean algorithm.

Specifically, if $f$ and $g$ are polynomials (i.e.,
$H \subset \{0,1,2, \cdots\})$ and if deg$(g)\leq$ deg$(f)$, the Euclidean
algorithm yields
\begin{equation}
\label{E:5.1.1}
  f(z)=g(z)q(z)+r(z)
\end{equation}
with deg$(r) < $ deg$(g)$.  We shall write
\begin{equation}
\label{E:5.1.2}
  q=quot(g,f), \quad \text{and} \quad r=rem(g,f).
\end{equation}

Since
\begin{equation}
\label{E:5.1.2}
  \begin{pmatrix}
    K & 0 \\
    0 & K^{-1}
  \end{pmatrix}
  \begin{pmatrix}
    1 & U \\
    0 & 1
  \end{pmatrix}
  =
  \begin{pmatrix}
    1 & K^{2}U \\
    0 & 1
  \end{pmatrix}
  \begin{pmatrix}
    K & 0 \\
    0 & K^{-1}
  \end{pmatrix},
\end{equation}
we may assume that the factor
\[
  \begin{pmatrix}
    K & 0 \\
    0 & K^{-1}
  \end{pmatrix}
\]
from the equation (\ref{E:5.1.2}) factorization occurs on the rightmost place.

\begin{equation}
\label{E:4.2.9}
  F=U_{N}[b],
\end{equation}
where $U$ is a unitary matrix-function, where
\[
  b=
  \begin{pmatrix}
    1 \\
    z \\
    z^{2} \\
    \vdots \\
    z^{N-1}
  \end{pmatrix}
\]
and where $U_{N}[b](z)=U(z^{N})b(z)$.

Let $U$ represent scalar valued matrix entry in a matrix function.
We now proceed to determine the polynomials $U_{1}(z), L_{1}(z), \cdots$,
etc. inductively starting with
\[
  A=
  \begin{pmatrix}
    1 & U \\
    0 & 1
  \end{pmatrix}
  B,
\]
where $U$ and $B$ are to be determined.  Introducing \ref{E:4.2.9}), this
reads
\begin{equation}
\label{E:5.1.3}
  A(z^{2})
  \begin{pmatrix}
    1 \\
    z
  \end{pmatrix}
  =
  \begin{pmatrix}
    1 & U(z^{2}) \\
    0 & 1
  \end{pmatrix}
  B(z^{2})
  \begin{pmatrix}
    1 \\
    z
  \end{pmatrix}
  =
  \begin{pmatrix}
    1 & U(z^{2}) \\
    0 & 1
  \end{pmatrix}
  \begin{pmatrix}
    h(z) \\
    k(z)
  \end{pmatrix}.
\end{equation}
But the matrix function
\[
  A=
  \begin{pmatrix}
    \alpha & \beta \\
    \gamma & \delta
  \end{pmatrix}
\]
is given and fixed see Remark \ref{R:4.2.6}.  Hence
\begin{equation}
\label{E:5.1.4}
  \gamma(z^{2})+\delta(z^{2})z=k(z)
\end{equation}
is also fixed.  The two polynomials to be determined are $u$ and $h$ in
(\ref{E:5.1.3}).  Carrying out the matrix product in (\ref{E:5.1.3}) yields:
\[
  \alpha(z^{2})+\beta(z^{2})z=h(z)+u(z^{2})k(z)
  = h_{0}(z)+h_{1}(z^{2})z+u(z^{2})\{\gamma(z^{2})+\delta(z^{2})z\}
\]
where we used the orthogonal splitting
\begin{equation}
\label{E:5.1.5}
  L^{2}(\mathbb{T})=S_{0}S_{0}^{*}L^{2}(\mathbb{T})\oplus
  S_{1}S_{1}^{*}L^{2}(\mathbb{T})
\end{equation}
from Lemma \ref{L:3.10}.  Similarly, from (\ref{E:5.1.4}), we get
\[
  \gamma(z^{2})+\delta(z^{2})z=k_{0}(z^{2})+k_{1}(z^{2})z;
\]
and therefore $\gamma = k_{0}$ and $\delta=k_{1}$, by Lemma \ref{L:4.1}.

Collecting terms and using the orthogonal splitting (\ref{E:5.1.5}) we arrive
at the following system of polynomial equations:
\begin{equation}
\label{E:5.1.6}
\begin{cases}
  \alpha = h_{0} + u\gamma \\
  \beta = h_{1} + u\delta ;
\end{cases}
\end{equation}
or more precisely,
\[
\begin{cases}
  \alpha(z) = h_{0}(z) + u(z)\gamma(z) \\
  \beta(z) = h_{1}(z) + u(z)\delta(z).
\end{cases}
\]
It follows that the two functions $u$ and $h$ may be determined from the
Euclidean algorithm.  With (\ref{E:5.1.2}), we get
\begin{equation}
\label{E:5.1.7}
\begin{cases}
  u=quot(\gamma, \alpha) \\
  h_{0}=rem(\gamma, \alpha) \\
  h_{1}=rem(\delta, \beta).
\end{cases}
\end{equation}

\begin{remark}
\label{R:5.1.1}
\cite{JoSo10}
The relevance of the determinant condition we have from Theorem \ref{T:4.2.5}
is as follows:
\[
  detA=\alpha\delta-\beta\gamma \equiv 1.
\]
Substitution of (\ref{E:5.1.6}) into this yields:
\[
  h_{0}\delta - h_{1}\gamma \equiv 1.
\]

Solutions to (\ref{E:5.1.6}) are possible because the two polynomials
$\delta(z)$ and $\gamma(z)$ are mutually prime.  The derived matrix
\[
  \begin{pmatrix}
    h_{0} & h_{1} \\
    \gamma & \delta
  \end{pmatrix}
\]
is obtained from $A$ via a row-operation in the ring of polynomials.

For the inductive step, it is important to note:
\begin{equation}
\label{E:5.1.8}
  deg(h_{0}) < deg(\gamma), \quad \text{and} \quad
  deg(h_{1}) < deg(\delta).
\end{equation}
The next step, continuing from (\ref{E:5.1.3}) is the determination of a
matrix-function $C$ and three polynomials $p, q,$ and $L$ such that
\begin{equation}
\label{E:5.1.9}
  \begin{pmatrix}
    1 & -U \\
    0 & 1
  \end{pmatrix}
  A=
  \begin{pmatrix}
    1 & 0 \\
    L & 1
  \end{pmatrix}
  C
\end{equation}
and
\begin{equation}
\label{E:5.1.10}
  \begin{pmatrix}
    1 & -U(z^{2}) \\
    0 & 1
  \end{pmatrix}
  A(z^{2})
  \begin{pmatrix}
    1 \\
    z
  \end{pmatrix}
  =
  \begin{pmatrix}
    1 & 0 \\
    L(z^{2}) & 1
  \end{pmatrix}
  \begin{pmatrix}
    p(z) \\
    q(z)
  \end{pmatrix}.
\end{equation}
Here
\[
  \begin{pmatrix}
    p \\
    q
  \end{pmatrix}
  =C(z^{2})
  \begin{pmatrix}
    1 \\
    z
  \end{pmatrix}.
\]
The reader will notice that in this step, everything is as before with the
only difference that now
\[
  \begin{pmatrix}
    1 & 0 \\
    L & 1
  \end{pmatrix}
\]
is lower diagonal in contrast with
\[
  \begin{pmatrix}
    1 & U \\
    0 & 1
  \end{pmatrix}
\]
in the previous step.

This time, the determination of the polynomial $p$ in (\ref{E:5.1.10}) is
automatic.  With
\[
  p(z)=p_{0}(z^{2})+zp_{1}(z^{2})
\]
(see (\ref{E:5.1.5})) and we get the following system:
\[
  \begin{cases}
    p_{0}=\alpha - u\gamma = h_{0} \\
    p_{1}=\beta - u\delta = h_{1} ; \quad \text{and}
  \end{cases}
\]
\[
  \begin{cases}
    \gamma=L(\alpha-u\gamma)+q_{0}=Lh_{0}+q_{0} \\
    \delta=L(\beta-u\delta)+q_{1}=Lh_{1}+q_{1} \\
  \end{cases}.
\]
So the determination of $L(z)$ and $q(z)=q_{0}(z^{2})+zq_{1}(z^{2})$ may be
done with Euclid:
\begin{equation}
\label{E:5.1.11}
  \begin{cases}
    L= quot(\alpha-u\gamma, \gamma)=quot(h_{0}, \gamma) \\
    q_{0}= rem(\alpha-u\gamma, \gamma)=rem(h_{0}, \gamma) \\
    q_{1}= rem(\beta-u\delta, \delta)=rem(h_{1}, \delta).
  \end{cases}
\end{equation}

Combining the two steps, the comparison of degrees is as follows:
\begin{equation}
\label{E:5.1.12}
  \begin{cases}
    deg(q_{0}) < deg(h_{0}) < deg(\gamma) \\
    deg(q_{1}) < deg(h_{1}) < deg(\delta)
  \end{cases}.
\end{equation}
Two conclusions now follow:
\begin{enumerate} [(i)]
  \item the procedure may continure by recursion;
  \item the procedure must terminate.
\end{enumerate}
\end{remark}

\begin{remark}
\label{R:5.1.2}
In order to start the algorithm in (\ref{E:5.1.7}) with direct reference to
Euclid, we must have
\begin{equation}
\label{E:5.1.13}
  deg(\gamma) \leq deg(\alpha)
\end{equation}
where
\[
  A=
  \begin{pmatrix}
    \alpha & \beta \\
    \gamma & \delta
  \end{pmatrix}
\]
is the initial $2 \times 2$ matrix-function.

Now, suppose (\ref{E:5.1.13}), i.e., that
\[
  deg(\gamma) > deg(\alpha).
\]
Then determine a polynomial $L$ such that
\begin{equation}
\label{E:5.1.14}
  deg(\gamma-L\alpha) \leq deg(\alpha).
\end{equation}
We may then start the procedure (\ref{E:5.1.7}) on the matrix function
\[
  \begin{pmatrix}
    \alpha & \beta \\
    \gamma -L\alpha & \delta
  \end{pmatrix}
  =
  \begin{pmatrix}
    1 & 0 \\
    -L & 1
  \end{pmatrix}
  A.
\]
If a polynomial $U$ and a matrix function $B$ is then found for
\[
  \begin{pmatrix}
    \alpha & \beta \\
    \gamma -L\alpha & \delta
  \end{pmatrix}
\]
then the factorization
\[
  A=
  \begin{pmatrix}
    1 & 0 \\
    L & 1
  \end{pmatrix}
  \begin{pmatrix}
    1 & U \\
    0 & 1
  \end{pmatrix}
  B
\]
holds; and the recursion will then work as outlined.

In the following, starting with a matrix-function $A$, we will always assume
that the degrees of the polynomials $(A_{i,j})_{i,j\in \mathbb{Z}_{N}}$ have
been adjusted this way, so the direct Euclidean algorithm can be applied.
\end{remark}

\subsection{The $3 \times 3$ case}
\label{sec:5.2}
The thrust of this section is the assertion that Theorem \ref{T:4.2.5} holds
with small modifications in the $3 \times 3$ case.

\subsubsection{Comments:}
In the definition of $A \in SL_{3}(\text{pol})$, it is understood that $A(z)$
has $detA(z)\equiv 1$ and that the entries of the inverse matrix $A(z)^{-1}$
are again polynomials.

Note that if $L, M, U$ and $V$ are polynomials, then the four matrices
\begin{equation}
\label{E:5.2.1}
  \begin{pmatrix}
    1 & 0 & 0 \\
    L & 1 & 0 \\
    0 & M & 1
  \end{pmatrix},
  \begin{pmatrix}
    1 & 0 & 0 \\
    0 & 1 & 0 \\
    L & 0 & 1
  \end{pmatrix},
  \begin{pmatrix}
    1 & U & 0 \\
    0 & 1 & V \\
    0 & 0 & 1
  \end{pmatrix}
  \quad \text{and} \quad
  \begin{pmatrix}
    1 & 0 & U \\
    0 & 1 & 0 \\
    0 & 0 & 1
  \end{pmatrix}
\end{equation}
are in $SL_{3}(\text{pol})$ since

\begin{equation}
\label{E:5.2.2}
  \begin{pmatrix}
    1 & 0 & 0 \\
    L & 1 & 0 \\
    0 & M & 1
  \end{pmatrix}^{-1}
  =
  \begin{pmatrix}
    1 & 0 & 0 \\
    -L & 1 & 0 \\
    LM & -M & 1
  \end{pmatrix} \quad \text{and}
\end{equation}

\begin{equation}
\label{E:5.2.3}
  \begin{pmatrix}
    1 & U & 0 \\
    0 & 1 & V \\
    0 & 0 & 1
  \end{pmatrix}^{-1}
  =
  \begin{pmatrix}
    1 & -U & UV \\
    0 & 1 & -V \\
    0 & 0 & 1
  \end{pmatrix}.
\end{equation}

\begin{theorem}
\label{T:5.2.1}
\cite{JoSo10}
Let $A \in SL_{3}(\text{pol})$; then the conclusion in Theorem \ref{T:4.2.5}
carries over with the modification that the alternating upper and lower
triangular matrix-functions now have the form (\ref{E:5.2.1}) or
(\ref{E:5.2.2})-(\ref{E:5.2.3}) where the functions $L_{j}, M_{j}, U_{j}$ and
$V_{j}$, $j=1,2, \cdots$ are polynomials.
\end{theorem}

\subsection{The $N \times N$ case}
\label{sec:5.3}

Below we outline the modifications to our algorithms from the 
$2 \times 2$ case needed in order to deal with filters with $N (> 2)$ bands, 
hence factorization of $N \times N$ matrix functions. The main difference when 
the number of frequency bands $N$ is more than $2$ is that in our 
factorizations, both the lower and the upper triangular factors, must take 
into account operations which cross between any pair of the total system of 
$N$ frequency bands.

\begin{theorem}
\label{T:5.3.1}
\cite{JoSo10}
Let $N \in \mathbb{Z}_{+}$, $N>1$, be given and fixed.  Let
$A \in SL_{N}(\text{pol})$; then the conclusions in Theorem \ref{T:4.2.5} carry
over with the modification that the alternative factors in the product are
upper and lower triangular matrix-functions in $SL_{N}(\text{pol})$.  We may
take the lower triangular matrix-factors
$\mathcal{L}=(L_{i,j})_{i,j \in \mathbb{Z}_{N}}$ of the form
\[
  \begin{pmatrix}
  1 & 0 & 0 & 0 & 0 & 0 & 0 & 0  \\
  0 & 1 & 0 & 0 & 0 & 0 & 0 & 0  \\
  L_{p} & 0 & 1 & 0 & 0 & 0 & 0 & 0 \\
  0     & L_{p+1} & 0  & 1 & 0 & 0 & 0 & 0 \\
  0 & 0 & . & 0  &1 & 0 & 0 & 0 \\
  0 & 0 & 0 & . & 0 & 1 & 0 & 0 \\
  0 & 0 & 0 & 0 & . & 0 & 1 & 0  \\
  0 & 0 & 0 & 0 & 0 & L_{N-1} & 0 & 1
  \end{pmatrix}
\]
polynomial entries
\begin{equation}
\label{E:5.3.1}
  \begin{cases}
  L_{i,i} \equiv 1, \\
  L_{i,j}(z)=\delta_{i-j, p}L_{i}(z);
  \end{cases}
\end{equation}
and the upper triangular factors of the form
$\mathcal{U}=(U_{i,j})_{i,j\in \mathbb{Z}_{N}}$ with
\begin{equation}
\label{E:5.3.2}
  \begin{cases}
  U_{i,i} \equiv 1, \\
  L_{i,j}(z)=\delta_{i-j, p}U_{i}(z).
  \end{cases}
\end{equation}
\end{theorem}
\begin{proof}
\textbf{Notation.}  Let $U_{1}, \cdots, U_{N}$, $L_{1}, \cdots, L_{N}$ be
polynomials and set
\begin{equation}
\label{E:5.3.3}
  \mathcal{U}_{N}(U)=
  \begin{pmatrix}
  1 & U_{1} & 0 & 0 & 0 & 0 & 0  \\
  0 & 1 & U_{2} & 0 & 0 & 0 & 0  \\
  0 & 0 & 1 & . & 0 & 0 & 0  \\
  0 & 0 & 0 & 1 & . & 0 & 0  \\
  0 & 0 & 0 & 0 & 1 & . & 0  \\
  0 & 0 & 0 & 0 & 0 & 1 & U_{N-1}  \\
  0 & 0 & 0 & 0 & 0 & 0 & 1
  \end{pmatrix}
\end{equation}

\begin{equation}
\label{E:5.3.4}
  \mathcal{L}_{N}(L)=
  \begin{pmatrix}
  1 & 0 & 0 & 0 & 0 & 0 & 0  \\
  L_{1} & 1 & 0 & 0 & 0 & 0 & 0  \\
  0 & L_{2} & 1 & 0 & 0 & 0 & 0  \\
  0 & 0 & . & 1 & 0 & 0 & 0  \\
  0 & 0 & 0 & . & 1 & 0 & 0  \\
  0 & 0 & 0 & 0 & . & 1 & 0  \\
  0 & 0 & 0 & 0 & 0 & L_{N-1} & 1
  \end{pmatrix}
\end{equation}

Note that both are in $SL_{N}(\text{pol})$; and we have
\[
  \mathcal{U}_{N}(U)^{-1}=\mathcal{U}_{N}(-U) \quad \text{and}
\]
\[
  \mathcal{L}_{N}(L)^{-1}=\mathcal{L}_{N}(-L).
\]

\textbf{Step 1:} Starting with $A=(A_{i,j})\in SL_{N}(\text{pol})$.  Then
left-multiply with a suitably chosen $\mathcal{U}_{N}(-U)$ such that the
degrees in the first column of $\mathcal{U}_{N}(-U)A$ decrease, i.e.,
\begin{equation}
\label{E:5.3.5}
  deg(A_{0,0}) \leq deg(A_{1,0}-u_{2}A_{1,0}) \leq \cdots deg(A_{N-1,0}).
\end{equation}
In the following, we shall use the same letter $A$ for the modified
matrix-function.

\textbf{Step 2:} Determine a system of polynomials $L_{1}, \cdots, L_{N-1}$
and a polynomial vector-function
\[
  \begin{bmatrix}
    f_{0} \\
    f_{1} \\
    \ldots \\
    f_{N-1}
  \end{bmatrix}
\]
such that
\begin{equation}
\label{E:5.3.6}
  A_{N}
  \begin{bmatrix}
    1 \\
    z \\
    z^{2} \\
    \ldots \\
    z^{N-1}
  \end{bmatrix}
  = \mathcal{L}_{N}(L)_{N}
  \begin{bmatrix}
    f_{0} \\
    f_{1} \\
    \ldots \\
    f_{N-1}
  \end{bmatrix},
\end{equation}
or equivalently
\[
  \sum_{j=0}^{N-1}A_{i,j}(z^{N})z^{j}=
  \begin{cases}
    f_{0}(z) &\text{ if $i=0$} \\
    L_{i}(z^{N})f_{i-1}(z)+f_{i}(z) &\text{ if $i>0$}
  \end{cases}.
\]

\textbf{Step 3:} Apply the operators $S_{j}$ and $S_{j}^{*}$ from section
\ref{sec:2} to both sides in (\ref{E:5.3.6}).  First (\ref{E:5.3.6}) takes
the form:
\[
  \sum_{j=0}^{N-1}S_{j}A_{i,j}=
  \begin{cases}
    f_{0} &\text{ if $i=0$} \\
    S_{f_{i-1}}L_{i}+f_{i} &\text{ if $i>0$}
  \end{cases}.
\]
For $i=1$, we get
\begin{equation}
\label{E:5.3.7}
  A_{1,j}=L_{1}A_{0,j}+k_{j} \quad \text{where} \quad k_{j}=S_{j}^{*}f_{1}.
\end{equation}

By (\ref{E:5.3.5}) and the assumptions on the matrix-functions, we note that
the system (\ref{E:5.3.7}) may now be solved with the Euclidean algorithm:
\begin{equation}
\label{E:5.3.8}
  \begin{cases}
    L_{1}=quot(A_{0,j}, A_{1,j}) \\
    k_{j}=rem(A_{0,j}, A_{1,j})
  \end{cases}
\end{equation}
with the same polynomial $L_{1}$ for $j=0,1, \cdots, N-1$.

For the polynomial function $f_{1}$ we then have
\begin{equation}
\label{E:5.3.9}
  f_{1}=\sum_{j=0}^{N-1}S_{j}k_{j};
\end{equation}
i.e.
\[
  f_{1}(z)=k_{0}(z^{N})+k_{1}(z^{N})z+\cdots+k_{N-1}(z^{N-1})z^{N-1}.
\]

The process now continues recursively until all the functions
$L_{1}, L_{2}, \cdots, f_{1}, f_{2}, \cdots$ have been determined.

\textbf{Step 4:} The formula (\ref{E:5.3.6}) translates into a
matrix-factorizations as follows: With $L$ and $F$ determined in
(\ref{E:5.3.6}), we get
\begin{equation}
\label{E:5.3.10}
  A=\mathcal{L}_{N}(L)B
\end{equation}
as a simple matrix-product taking $B=(B_{i,j})$ and
\begin{equation}
\label{E:5.3.11}
  B_{i,j}=S_{j}^{*}f_{i},
\end{equation}
where we used Lemmas \ref{L:3.10} and \ref{L:4.1}.

\textbf{Step 5:} The process now continues with the polynomial matrix-function
from (\ref{E:5.3.10}) and (\ref{E:5.3.11}).  We determine polynomials
$U_{1}, \cdots, U_{N-1}$ and a third matrix function
\[
  C=(C(z))=(C_{i,j}(z)) \quad \text{such that} \quad B=\mathcal{U}_{N}(U)C.
\]

\textbf{Step 6:} As each step of the process we alternate $L$ and $U$; and
at each step, the degrees of the matrix-functions is decreased.  Hence the
recursion must terminate as stated in Theorem \ref{T:5.3.1}.
\end{proof}

\subsection{$L^{\infty}(\mathbb{T})$-matrix entries.}
\label{sec:3.1}
While the case $N=2$ is motivated by application to the high-pass v.s.
low-pass filters, may result for the $N=2$ case carry over.  To see this, we
first define the Cuntz-algebra $\mathcal{O}_{N}$ in general the relations are
\begin{equation}
\label{eq:3.1.1}
  S_{i}^{*}S_{j}=\delta_{i, j}I, \quad \sum_{i}S_{i}S_{i}^{*}=I,
\end{equation}
when the elements $(S_{i})_{i=0}^{N-1}$ are given symmetrically.

Each case (\ref{eq:3.1.1}) has many representations; for example if
$(m_{i}(z))_{i=0}^{N-1}$, $z \in \mathbb{T}$, is a system of filters
corresponding to $N$ frequency bands, we may obtain a representation of
$\mathcal{O}_{N}$ acting on the Hilbert space $L^{2}(\mathbb{T})$ as
follows
\begin{equation}
\label{eq:3.1.2}
  (S_{i}\psi)(z)=m_{i}(z)\psi(z^{N}), \quad \forall z \in \mathbb{T},
  \quad \psi \in L^{2}(\mathbb{T}).
\end{equation}

For $i \in \{0, 1, \cdots, N-1 \}$, the adjoint operator of $S_{i}$ in
(\ref{eq:3.1.2}) is
\begin{equation}
\label{eq:3.1.3}
  (S_{i}^{*}\psi)(z)=\frac{1}{N}\sum_{w^{N}=z}\overline{m_{i}(z)}\psi(z^{N}),
  \quad z \in \mathbb{T}.
\end{equation}
A direct verification shows that the Cuntz-relation (\ref{eq:3.1.1}) are
satisfied for the operators $(S_{i})_{i=0}^{N-1}$ in (\ref{eq:3.1.2}) if
and only if the system $(m_{i})_{i=0}^{N-1}$ is a multi-band filter
covering the $N$ frequency bands.

The simplest example of the representation in (\ref{eq:3.1.2}) is the case
where $m_{i}(z)=z^{i}$, $i=0,1, \cdots, N-1$; and so
\begin{equation}
\label{eq:3.1.4}
  (S_{i}\psi)(z)=z^{i}\psi(z^{N}), \quad i=0,1, \cdots, N-1,
  \quad z \in \mathbb{T}, \quad \psi \in L^{2}(\mathbb{T}).
\end{equation}

\begin{theorem}
\label{T:3.1.1}
Let $g=(g_{ij})_{i,j=0}^{N-1}\in SL_{N}(L^{\infty}(\mathbb{T})$, i.e.,
$g_{ij}(\cdot) \in L^{\infty}(\mathbb{T}))$, and
\begin{equation}
\label{eq:3.1.5}
  det g(\cdot)\equiv 1 \quad \text{ on $\mathbb{T}$}
\end{equation}
then for every factorization
\begin{equation}
\label{eq:3.1.6}
  g(z)=
  \begin{pmatrix}
    1 & 0 & 0 & \cdots & 0 \\
    L_{1}(z) & 1 & 0 & \cdots & 0 \\
    L_{2}(z) & 0 & 1 & \ddots & 0 \\
    \vdots & \vdots & \vdots & \ddots & \vdots \\
    \vdots & \vdots & \vdots & \ddots &\vdots \\
    L_{N-1}(z) & 0 & \cdots & 0 & 1
  \end{pmatrix}
  g^{(new)}(z) \quad \text{(matrix product)}
\end{equation}
there is a unique $f_{i} \in L^{\infty}(\mathbb{T})$ such that
\begin{equation}
\label{eq:3.1.7}
  g_{0,j}^{(new)}(z)=g_{0,j}(z), \quad \text{and}
\end{equation}
\begin{equation}
\label{eq:3.1.8}
  g_{i,j}^{(new)}(z)=S_{j}^{*}f_{i}, \quad \text{for $i=1,2, \cdots, N-1$}
\end{equation}
where $\{S_{j}\}_{j=0}^{N-1}$ is the system of Cuntz-isometries from
(\ref{eq:3.1.4}).
\end{theorem}
\begin{proof}
With the arguments above, in the space $\mathcal{O}_{N}$ of $N=2$, we now
get matrix, the system:
\begin{equation}
\label{eq:3.1.9}
  g^{(new)}(z^{N})
  \begin{pmatrix}
  1 \\
  z \\
  z^{2} \\
  \vdots \\
  z^{N-1}
  \end{pmatrix}
  =
  \begin{pmatrix}
  f_{0}(z) \\
  f_{1}(z) \\
  f_{2}(z) \\
  \vdots \\
  f_{N-1}(z)
  \end{pmatrix},
\end{equation}

\begin{equation}
\label{eq:3.1.10}
  \begin{cases}
    S_{j}^{*}f_{0}=g_{0,j}, \\
    L_{i}g_{0,j}+S_{j}^{*}f_{i}=g_{i,j},
  \end{cases}
\end{equation}
and
\begin{equation}
\label{eq:3.1.11}
  f_{i}=\sum_{j=0}^{N-1}S_{j}S_{j}^{*}f_{i}
  =\sum_{j=0}^{N-1}S_{j}(g_{i,j}-L_{i}g_{0,j})
\end{equation}
for $i=1,2, \cdots, N-1$, which is desired conclusion.
\end{proof}

\subsection{Optimal factorization in the case of
$SL_{N}(L^{\infty}(\mathbb{T}))$}
\label{sec:3.2}
Fix $N>2$, and consider the usual inner product in $\mathbb{C}^{N}$,
\begin{equation}
\label{eq:3.2.1}
  \langle z, w \rangle:=\sum_{j=0}^{N-1}\overline{z_{j}}w_{j},
\end{equation}
defined for all $z=(z_{0}, \cdots, z_{N-1})$, and $w=(w_{0}, \cdots, w_{N-1})$.

For $g=(g_{ij}(z))_{i,j=0}^{N-1} \in SL_{N}(L^{\infty}(\mathbb{T}))$, set
\[
  \tilde{g_{0}}(z)=(g_{0j}(z))_{j=0}^{N-1}, \quad a.e.,
\]
the first row in the matrix-function
$\mathbb{T} \ni Z \mapsto (g(z)) \in SL_{N}(L^{\infty}(\mathbb{T}))$.  Let
$P(z)=P^{(g)}(z)$ denote the projection of $\mathbb{C}^{N}$ onto the
one-dimensional subspace generated by $\tilde{g_{0}}(z) \in \mathbb{C}^{N}$.

Note that $(P(z))_{z \in \mathbb{T}}$ is a field of orthogonal rank-$2$
projection in $\mathbb{C}^{N}$.  Setting
\begin{equation}
\label{eq:3.2.2}
  \|\tilde{g_{0}}(z)\|_{2}^{2}=\sum_{j=0}^{N-1}|g_{0,j}(z)|^{2},
\end{equation}
we have:
\begin{equation}
\label{eq:3.2.3}
  P(z)\xi=\sum_{j=0}^{N-1}\frac{\overline{g_{0,j}(z)}\xi_{j}}
  {\|\tilde{g_{0}}(z)\|_{2}^{2}}\text{ }g_{0,j}(z)
  \quad \text{for all } \xi=(\xi{0}, \cdots, \xi_{N-1}) \in \mathbb{C}^{N};
\end{equation}
and set
\begin{equation}
\label{eq:3.2.4}
  \tilde{g_{j}}^{(new)}(z)=\tilde{g_{0}}(z)-P(z)\tilde{g_{j}}(z).
\end{equation}

\begin{corollary}
\label{C:3.2.1}
\begin{enumerate} [(i)]
  \item For the factorization (\ref{eq:3.1.6}) in Theorem \ref{T:3.1.1}, the
optimal choice is that given by the matrix-factor $f^{(new)}$ having as rows
the vector fields $\tilde{g_{i}}^{(new)}(z)$ specified in (\ref{eq:3.2.4}).
  \item With the resolution of row-vector fields,
\begin{equation}
\label{eq:3.2.5}
  \tilde{g_{j}}^{(new)}(z)=(S_{0}^{*}f_{i}, S_{1}^{*}f_{i}, \cdots,
  S_{N-1}^{*}f_{i})
\end{equation}
from (\ref{eq:3.1.8}), the optimal solution is attained; and it is the unique
minimizer for the following system of optimization problems:
\begin{equation}
\label{eq:3.2.6}
  min_{f_{i}\in L^{2}(\mathbb{T})}\|f_{i}\|_{L^{2}(\mathbb{T}}^{2}, \quad
  1 \leq i < N,
\end{equation}
where each choice $(f_{i})_{i=1}^{N-1}$ yields a matrix-factor
$\mathcal{A}^{(new)}$ via (\ref{eq:3.2.5}).
\end{enumerate}
\end{corollary}
\begin{proof}
  The proof of the conclusions in (i)-(ii) in the corollary follows from the
arguments in the proof of Theorem \ref{T:3.1.1} above.
\end{proof}





\end{document}